\documentclass[3p]{elsarticle}

\usepackage{amsmath}
\usepackage{amssymb}
\usepackage{amsthm}

\def\fpd#1#2{{\displaystyle\frac{\partial #1}{\partial #2}}}
\def\spd#1#2#3{{\displaystyle\frac{\partial^2 #1}
{\partial #2\partial #3}}}
\def\lie#1{{\cal L}_{#1}}
\def\vf#1{{\displaystyle\frac{\partial }{\partial #1}}}

\def\onehalf{{\textstyle\frac12}}

\def\conn#1#2#3{\setbox1=\hbox{$\scriptstyle{#2}{#3}$}%
\setbox2=\hbox to\wd1{$\hfil\scriptstyle{#1}\hfil$}
\Gamma^{\!\box2}_{\!\box1}}

\def\TMO{T^\circ\!M}
\def\TxMO{T_x^\circ\!M}
\def\TNO{T^\circ\!N}

\def\clift#1{#1^{\scriptscriptstyle{\mathrm{C}}}}
\def\hlift#1{#1^{\scriptscriptstyle{\mathrm{H}}}}
\def\vlift#1{#1^{\scriptscriptstyle{\mathrm{V}}}}

\def\rk{\mathop{\mathrm{rank}}}

\def\D{\mathcal{D}}
\def\J{\mathfrak{J}}
\def\L{\mathfrak{L}}
\newcommand{\R}{\mathbb{R}}
\def\S{\mathfrak{S}}

\renewcommand{\a}{\alpha}
\renewcommand{\b}{\beta}
\newcommand{\Dt}{\tilde{\Delta}}
\newcommand{\K}{\mathcal{K}}
\renewcommand{\P}{\mathrm{P}}
\newcommand{\Se}{Segre}
\newcommand{\To}{T^\circ}
\newcommand{\trac}{\mathcal{T}M}
\newcommand{\traco}{\mathcal{T}^\circ\!M}
\newcommand{\TVO}{\To\kern-0.08em(\VM)}
\newcommand{\U}{\Upsilon}
\newcommand{\Ut}{\tilde{\U}}
\newcommand{\V}{\mathcal{V}}
\newcommand{\VM}{\V\kern-0.08em M}

\newtheorem{thm}{Theorem}
\newtheorem{lem}{Lemma}
\newtheorem{prop}{Proposition}
\newtheorem{cor}{Corollary}

\begin{document}

\begin{frontmatter}
\title{Hilbert forms for a Finsler metrizable projective class of sprays}

\author{M.\ Crampin, T.\ Mestdag}
\address{Department of Mathematics, Ghent University,
Krijgslaan 281, B--9000 Gent, Belgium}
\author{and D.\,J.\ Saunders}
\address{Department of Mathematics, The University of Ostrava,
30.\ dubna 22, 701 03 Ostrava, Czech Republic}

\begin{abstract}\noindent
The projective Finsler metrizability problem deals with the
question whether a projective-equivalence class of sprays is the
geodesic class of a (locally or globally defined) Finsler
function. In this paper we use Hilbert-type forms to state a
number of different ways of specifying necessary and sufficient
conditions for this to be the case, and we show that they are
equivalent. We also address several related issues of interest
including path spaces, Jacobi fields, totally-geodesic
submanifolds of a spray space, and the equivalence of path
geometries and projective-equivalence classes of sprays.
\end{abstract}


\end{frontmatter}

\section{Introduction}

A Finsler function can in many ways be regarded as a singular
Lagrangian. As such, there are many sprays whose base integral
curves are solutions of the Euler-Lagrange equations of a given
Finsler function. These sprays are all projectively equivalent
and together they constitute the geodesic class of sprays of the
given Finsler function. It is therefore natural to ask whether
or not a given projective-equivalence class (or projective
class, for short) of sprays is the geodesic class of some
Finsler function, or, in the terminology of this paper, whether
or not a projective class of sprays is projectively Finsler
metrizable.

One may think of (at least) three approaches to formulating
the necessary and sufficient conditions for this to be the case.
They differ with respect to what kind of geometric object the
conditions are expressed in terms of:
\begin{enumerate}
\item a multiplier, that is, a symmetric twice covariant tensor along
the tangent bundle projection $\tau$, leading to Helmholtz-like conditions;
\item a semi-basic 1-form, leading to the conditions given by
Bucataru and Muzsnay \cite{BM} for such a form to be a Hilbert 1-form;
\item a 2-form, leading to conditions for such a form to be a
Hilbert 2-form.
\end{enumerate}
The third item can be further subdivided:
\begin{enumerate}
\item[3.1] the 2-form is given on the slit tangent bundle,
leading to conditions similar to those given for the `ordinary'
inverse problem of the calculus of variations by the first
author as long ago as 1981 \cite{MCold};
\item[3.2] the 2-form is given on a certain manifold on which is
defined an almost Grassmann structure associated with the
projective class, leading to conditions formulated by the first
and third authors in \cite{Japan};
\item[3.3] the 2-form is given on path space, leading to conditions
discussed by  \'{A}lvarez Paiva in \cite{AP}.
\end{enumerate}

Note that unlike \'{A}lvarez Paiva, who in \cite{AP} deals
only with reversible paths, that is, paths which have no
preferred orientation, we cover in this paper the more general case of
oriented paths, or sprays in the fully general sense.

We have discussed the multiplier approach in detail in
\cite{mult}. In this paper we deal with the versions of the
conditions involving forms, that is, items 2 and 3.1--3.3 of the lists
above.

It might be argued that there are two additional approaches that
should be taken into account. One is the use of the Rapcs\'{a}k
conditions (which are discussed in \cite{Shen,Szi} for example).
We prefer to think of these conditions as just being
reformulations of the Euler-Lagrange equations. They do play a
significant role in our analysis of the multiplier problem, and
have been discussed in \cite{mult}. The other is the holonomy
method described in \cite{hol}. This approach is well-suited to
the problem of determining whether a given spray is the
canonical spray of a Finsler function, that is, the one whose
integral curves are parametrized (up to affine transformations)
by arc-length. However, it is not easily adapted to the
projective problem which is the subject of this paper. We do not
consider it further here therefore.

To the best of our knowledge, this paper states for the first
time the metrizability conditions in terms of 2-forms on the
slit tangent bundle. We also address the global aspects of the
problem. The main purpose of this paper, however, is to discuss
the relationship between the various approaches enumerated
above, and in particular to show that they are equivalent. Such
a discussion is in particular needed because comparison of the
different results in the literature is far from obvious. To give
just one example:\ whereas most authors consider the projective
class of sprays as the main object under investigation, others,
in particular \'{A}lvarez Paiva, give priority to the paths. We
have therefore considered it desirable to discuss the
relationship between what is called by \'{A}lvarez Paiva in
\cite{AP} a path geometry, and a projective class of
sprays. In the course of the discussion it will also be necessary
to address a number of issues related to Finsler geometry and
the projective geometry of sprays which are of interest in their
own right, including Jacobi fields and totally-geodesic
submanifolds of a spray space. We express our results as far as
possible in projectively-invariant terms; in particular, this
means that throughout we use the Finsler function rather than
the energy, and avoid reference to the canonical geodesic spray.
In the terminology introduced in \cite{SzVa} we deal entirely
with the problem of metrizability in the broad sense.

The paper begins with a version of \'{A}lvarez Paiva's definition
of a path geometry adapted to the concerns of this paper. We
show that in fact there is no loss of generality in working with
sprays.

In Section 3 we give a summary of the relevant results on the
multiplier problem from \cite{mult}. In Section 4 we quote the
theorem of Bucataru and Muzsnay mentioned in item 2 above, and
show that the conditions it contains are equivalent to those
that must be satisfied by a multiplier. In Section 5 we give the
most straightforward of the formulations of the conditions in
terms of the existence of a 2-form with certain properties, and
in the following section the somewhat more sophisticated version
in which the 2-form is specified on a certain manifold which
carries an almost Grassmann structure associated with a given
projective class of sprays.

All three of the versions of the conditions discussed in
Sections 4--6 involve closed 2-forms of which the involutive
distribution $\D$ determined by the projective class (see the next section)
is the characteristic distribution. A natural further step
therefore is to quotient out by $\D$, as one might say. Where
this is possible the manifold obtained is called the path space,
since each of its points represents a geodesic path of the
projective class. The 2-form in question passes to the quotient
to define a symplectic form there. In Section 7 we elaborate on
this construction and begin the discussion of the further
properties of the symplectic structure. As we show in Section
8, tangent vectors to path space can be thought of as Jacobi
fields. Using this insight we reformulate the positive
quasi-definiteness property of the multiplier required for the
local existence of a Finsler function.

One much discussed special case of the projective metrizability
problem is that raised by the Finslerian version of Hilbert's
fourth problem; this indeed is the main subject of \cite{AP}. In
\'{A}lvarez Paiva's analysis an important role is played by
2-planes in $\R^n$. From the more general point of view adopted
here what is significant about planes in $\R^n$ is that they are
totally-geodesic submanifolds. We develop a theory of
totally-geodesic submanifolds of spray manifolds in Section 9, and use
it to give a modest generalization of one of the results of
\cite{AP}. The paper ends with an illustrative example.

\section{Path geometries and sprays}

We first recall some basic concepts from spray and Finsler geometry, mainly to fix notations.

We shall always assume that the base manifold $M$ is smooth and
paracompact. Unless it is explicitly stated otherwise, we assume
that $\dim M\geq 3$. The slit tangent bundle of $M$ is the
tangent bundle with the zero section removed. We shall denote it
by $\tau:\TMO \to M$.

A {\em spray} is a vector field on $\TMO$ such that
$\tau_*\Gamma_{(x,y)}=y$ for any $x\in M$ and $y\in T_xM$,
$y\neq 0$, and such that $[\Delta,\Gamma]=\Gamma$ where $\Delta$
is the Liouville field. It is locally of the form
\[
\Gamma=y^i\vf{x^i}-2\Gamma^i\vf{y^i}
\]
and it determines a horizontal distribution, spanned by the vector fields
\[
H_i=\vf{x^i}-\Gamma^j_i\vf{y^j},\quad
\Gamma^i_j=\fpd{\Gamma^i}{y^j}.
\]
We shall also write $V_i$ for the vertical vector fields
$\partial/\partial y^i$. Horizontal and vertical lifts of a
vector field $X$ on $M$ are denoted by $\hlift{X}$ and
$\vlift{X}$, respectively.

Two sprays are said to be {\em projectively equivalent} if their
geodesics (base integral curves) are the same up to an
orientation-preserving reparametrization. The geodesics of
projectively equivalent sprays, in other words, define oriented
paths in $M$. Projective equivalence is an equivalence relation
on sprays; an equivalence class is called a projective class of
sprays. If $\Gamma$ is a spray, then any member of its
projective class takes the form $\Gamma-2P\Delta$ for some
function $P$ on $\TMO$ which satisfies $\Delta(P)=P$. Note that
a projective class of sprays determines an involutive
two-dimensional distribution $\D$ on $\TMO$, which is spanned by
$\Delta$ and any spray $\Gamma$ of the class. This distribution
plays an important role in our analysis. We refer to e.g.\
\cite{Shen} for further reading on the geometry of sprays.

We shall work throughout with projective classes of
sprays. It might however be regarded as more natural from the
geometrical point of view to see a projective class of sprays as
merely a surrogate for the collection of its geodesic paths, and
to think of the metrizability problem as the question of whether
a collection of oriented paths, suitably specified, is the set
of geodesic paths of a Finsler structure. \'{A}lvarez Paiva for
example, in \cite{AP} Section 4, has taken such an idea as basic and
formalised it into the concept of a path geometry. In this
section we give a definition of path geometry based on
\'{A}lvarez Paiva's, but differing from his in that it deals
with oriented paths; and we show that there is no loss of
generality in working with sprays.

For any smooth manifold $M$ we denote by $\sigma:STM\to M$ its
sphere bundle, that is, the quotient of $\TMO$ by the action
induced by $\Delta$, so that a point $s$ of $STM$ is an
equivalence class $[y]$ of vectors $y$ in $T_xM$, $x=\sigma(s)$,
where the equivalence relation is multiplication by a positive
scalar. A {\em path geometry\/} on $M$ is a smooth foliation of
$STM$ by oriented one-dimensional submanifolds $\S$ which satisfy
what one might call the second-order property, namely that if
$\S_s$ is the submanifold through $s$, the (oriented) tangent
space to $\sigma(\S_s)$ at $x$ coincides with $[y]$ (where
$s=[y]$).

We define a distribution $\D$ on $\TMO$ as follows:\ $v\in\D_y
\subset T_y\TMO$ if the projection of $v$ to $STM$ is
tangent to $\S_{[y]}$. Then $\D$ is an involutive
two-dimensional smooth distribution on $\TMO$, containing
$\Delta$. We shall show that $\D$ is the distribution
corresponding to a projective class of sprays on $\TMO$.

\begin{thm}
For any given path geometry on $STM$, there is a
projective class of sprays on $\TMO$ such that the
distribution $\D$ is spanned by $\Delta$ and any spray of the
class.
\end{thm}

\begin{proof}
We have to construct a suitable spray $\Gamma$.

There is a covering of $\TMO$ by open sets $U$, which we may
assume to be connected, such that on $U$ there is a smooth
vector field $Z_U$ such that $\D|_U$ is the span of $\Delta$ and
$Z_U$. The projection of $Z_U$ to $STM$ is tangent to the
foliation, and never vanishes. We may assume that it is oriented
positively with respect to the foliation. Then for every
$(x,y)\in U$, $\tau_*Z_U(x,y)$ is a positive scalar multiple of
$y$, say $\tau_*Z_U(x,y)=\zeta(x,y)y$ where $\zeta$ is a
positive smooth function on $U$. Set
$\tilde{\Gamma}_U=(1/\zeta)Z_U$; then $\tilde{\Gamma}_U$ is a
second-order differential equation field on $U$, and $\D|_U$ is
the span of $\Delta$ and $\tilde{\Gamma}_U$.

The manifold $M$, which is assumed to be paracompact, admits a global
Riemannian metric, say $g$. Denote by $G$ the function on $\TMO$ given by the
Riemannian norm, so that $G(x,y)=\sqrt{g_x(y,y)}$. Note that
$\Delta(G)=G$. We can change the local basis of  $\D|_U$ by adding
some scalar multiple of $\Delta$ to $\tilde{\Gamma}_U$, and we
can do so in such a way that the new vector field
$\Gamma_U=\tilde{\Gamma}_U+f\Delta$ satisfies $\Gamma_U(G)=0$:\
just take $f=-\tilde{\Gamma}_U(G)/G$. Of course, $\Gamma_U$ is
also a second-order differential equation field. It is moreover
uniquely determined by the properties that it is a second-order
differential equation field in $\D|_U$ and satisfies
$\Gamma_U(G)=0$: for if $\Gamma'_U$ also has those properties
then $\Gamma_U-\Gamma'_U$ is vertical, in $\D|_U$, and therefore
a scalar multiple of $\Delta$; but since
$\Gamma_U(G)-\Gamma'_U(G)=0$, while $\Delta(G)=G$, the scalar
factor must be zero.

It follows that there is a globally-defined vector field
$\Gamma$, which is a second-order differential equation field in
$\D$ satisfying $\Gamma(G)=0$, such that $\Gamma_U=\Gamma|_U$.
For if $\Gamma_U$ and $\Gamma_{U'}$ are the unique local vector
fields with those properties on $U$ and $U'$ then by uniqueness
they must agree on $U\cap U'$.

Finally, we show that $\Gamma$ is a spray, that is, that it
satisfies $[\Delta,\Gamma]=\Gamma$. Now $[\Delta,\Gamma]-\Gamma$
is certainly vertical, simply because $\Gamma$ is a second-order
differential equation field. Thus
$[\Delta,\Gamma]=\Gamma+f\Delta$ for some function $f$ on
$\TMO$. But $\Gamma(G)=0$, and
$[\Delta,\Gamma](G)=\Delta(\Gamma(G))-\Gamma(\Delta(G))=-\Gamma(G)=0$;
but $\Delta(G)=G$, and so $f=0$.
\end{proof}

\section{Some results on the multiplier problem}
In order to keep the paper more or less self-contained, we shall
quote here some results from \cite{mult}.

A {\em Finsler function} is a smooth function on $\TMO$, which
is positive, positively (but not necessarily absolutely)
homogeneous, and strongly convex. The last property means that
the matrix of functions
\[ g_{ij} = \spd{F^2}{y^i}{y^j}
\]
must be positive definite. The {\it Hilbert 1- and 2-forms} on $\TMO$
are given, respectively, by
\[
\theta = \fpd{F}{y^i} dx^i \quad \mbox{and}\quad d\theta.
\]
We shall say that $\Gamma$ is a geodesic spray for $F$ if its
base integral curves are solutions of the Euler-Lagrange
equations of $F$. The set of geodesic sprays for $F$ form a
projective class. A modern introduction to Finsler geometry can
be found in \cite{Bao}.

We shall use the term {\em multiplier} for a (0,2) tensor field
$h$ along the slit tangent bundle projection $\tau$. A
multiplier will also be called a tensor or tensor field for
short, and we shall often denote it simply by its components
$h_{ij}(x,y)$.

The conditions on a multiplier that form the basis of the
analysis in \cite{mult} are these:
\begin{align*}
h_{ji}&=h_{ij}\\
h_{ij}y^j&=0\\
\fpd{h_{ij}}{y^k}&=\fpd{h_{ik}}{y^j}\\
(\nabla h)_{ij}&=0\\
h_{ik}W^k_j&=h_{jk}W^k_i.
\end{align*}
\noindent Here $\nabla$ stands for the dynamical covariant
derivative operator of any choice of spray $\Gamma$ in the
projective class. The action of this operator on tensors $h$ is
given by
$(\nabla h)_{ij}=\Gamma(h_{ij})-\Gamma_i^kh_{kj}-\Gamma_j^kh_{ik}$.
The functions $W^k_j$ are the components of the
(projectively-invariant) Weyl tensor. A result of \cite{mult}
states that the last condition can equivalently be replaced by
the condition $\oplus R^l_{jk}h_{il}=0$, where $\oplus$ stands
for a cyclic sum and where $R^l_{jk}$ are the curvature
components of the horizontal distribution (with $[H_i,H_j] =-
R^l_{ij} V_l$), or by the condition $h_{ik}R^k_j=h_{jk}R^k_i$,
where $R^k_j= R^k_{jl}y^l$ are the components of the Jacobi
endomorphism (Riemann curvature).

The conditions displayed above, though expressed in coordinate
form, are tensorial in nature. They play the same role in
relation to the projective Finsler metrizability problem as the
Helmholtz conditions do for the general inverse problem of the
calculus of variations; though it is not strictly accurate, for
ease of reference we shall call them the Helmholtz conditions in
this paper (in \cite{mult} we referred to them as Helmholtz-like
conditions).

A tensor $h_{ij}$ is said to be {\em positive quasi-definite} if
$h_{ij}(y)v^iv^j\geq0$, with equality only if $v$ is a scalar
multiple of $y$. We shall say that a multiplier $h$ is {\em
quasi-regular\/} if $h_{ij}(y)v^j=0$ if and only if $v^i=ky^i$
for some scalar $k$. We shall call a positively-homogeneous
function $F$ whose Hessian with respect to fibre coordinates is
quasi-regular a {\em pseudo-Finsler function}. We summarize the
relevant results from \cite{mult} in the following theorem (they
occur as Theorems 2, 3 and 4 in \cite{mult}).

\begin{thm}\label{globmult}
(1) Given a projective class of sprays over a manifold $M$, and any
contractible coordinate neighbourhood $U\subset M$, there is a
positively-homogeneous function $F$ on $T^\circ U$ such that every
spray in the class satisfies the Euler-Lagrange equations for
$F$ if and only if there are functions $h_{ij}$ on  $T^\circ U$ which satisfy the Helmholtz conditions.\\
(2) If $F$ is a (global) Finsler function on $\TMO$ then its Hessian
$h$ satisfies the Helmholtz conditions for the
sprays of its geodesic class, and is in addition positive
quasi-definite. Conversely, suppose given a
projective class of sprays on $\TMO$. If there is a
tensor field $h$ on $\TMO$ which satisfies the Helmholtz
conditions everywhere and is in addition positive
quasi-definite, and if $H^2(M)=0$, then the projective class is
the geodesic class of a global pseudo-Finsler function, and of a
local Finsler function over a neighbourhood of any point of $M$.\\
(3) The projective class of a reversible spray on $\TMO$
is the geodesic class of a globally-defined absolutely-homogeneous
Finsler function if and only if there is a tensor field
$h$ which satisfies the Helmholtz conditions everywhere and is
in addition positive quasi-definite.
\end{thm}

\section{The theorem of Bucataru and Muzsnay}\label{BMthm}

The following theorem appears in \cite{BM}.

\begin{thm}[Bucataru and Muzsnay]\label{bm}
A spray $\Gamma$ is projectively metrizable if and only if there
exists a semi-basic 1-form $\theta$ on $\TMO$ such that
\[
\rk(d\theta)=2n-2,\quad i_\Gamma\theta>0,
\]
\[
\lie{\Delta}\theta=0,\quad d_J\theta=0, \quad d_H\theta=0.
\]
\end{thm}

We have modified the notation to fit ours. Here $J$ is the tangent
structure and $H$ the horizontal projector, both type $(1,1)$ tensor
fields on $\TMO$:
\[
J=V_i\otimes dx^i,\quad H=H_i\otimes dx^i.
\]
The conditions $d_J\theta=0$ and $d_H\theta=0$ amount to
\[
d\theta(\vlift{X},\hlift{Y})+d\theta(\hlift{X},\vlift{Y})=0,\quad
d\theta(\hlift{X},\hlift{Y})=0
\]
respectively, where $X$ and $Y$ are any vector fields on $M$; or in
terms of the basis fields,
\[
V_i(\theta_j)=V_j(\theta_i),\quad H_i(\theta_j)=H_j(\theta_i),
\quad\mbox{where } \theta=\theta_i dx^i.
\]

We call the conditions in the first line of the theorem the algebraic
conditions, those in the second line the differential
conditions, on $\theta$. We show first that the differential
conditions are equivalent to the Helmholtz conditions.

\begin{thm}\label{bmmult}
Suppose that, for a given spray $\Gamma$, there is a semi-basic 1-form
$\theta$ satisfying the differential conditions of Theorem~\ref{bm}.
Then $h_{ij}=V_i(\theta_j)$ satisfies the Helmholtz conditions.
Conversely, suppose that the tensor $h_{ij}$
satisfies  the Helmholtz conditions.  Then there is a
semi-basic 1-form $\theta$ which satisfies the differential conditions
of Theorem~\ref{bm}, and $h_{ij}=V_i(\theta_j)$.
\end{thm}

\begin{proof}
Suppose that the semi-basic 1-form $\theta$ satisfies the differential
conditions of Theorem~\ref{bm}.  Set $h_{ij}=V_i(\theta_j)$.  This is
a tensor field along $\tau$ of the indicated type.  Since
$V_i(\theta_j)=V_j(\theta_i)$, as we pointed out above, $h_{ij}$ is
symmetric.  Moreover
$V_k(h_{ij})=V_kV_j(\theta_i)=V_jV_k(\theta_i)=V_j(h_{ik})$, since
$V_j$ and $V_k$ commute.  Furthermore
$\lie{\Delta}\theta=y^jV_j(\theta_i)dx^i=h_{ij}y^jdx^i=0$.  Now
$H_i(\theta_j)=H_j(\theta_i)$, from which it follows that
\[
\Gamma(\theta_i)=H_i(\theta_k)y^k=H_i(\theta_ky^k)+\Gamma^k_i\theta_k.
\]
Now apply $V_j$  and use $[V_j,\Gamma]=H_j-\Gamma_j^kV_k$ to obtain
\[
V_j\Gamma(\theta_i)=\Gamma(h_{ij})+H_j(\theta_i)-\Gamma^k_jh_{ik}
=V_jH_i(\theta_ky^k)+\conn kij\theta_k+\Gamma^k_ih_{jk},
\]
where $\Gamma^k_{ij} = V_j(\Gamma^k_i)$. But $[V_j,H_i]=-\conn kijV_k$, and
$V_j(\theta_ky^k)=h_{jk}y^k+\theta_j=\theta_j$. Thus
$V_jH_i(\theta_ky^k)+\conn ijk\theta_k=H_i(\theta_j)$, and so
\[
\Gamma(h_{ij})-\Gamma^k_jh_{ik}-\Gamma^k_ih_{jk}
=(\nabla h)_{ij}=0.
\]
Finally, note that $[H_j,H_k](\theta_i)=-R_{jk}^lh_{il}$:\ but then
\[
\oplus [H_j,H_k](\theta_i)=-\oplus R_{jk}^lh_{il}=0
\]
in virtue of the fact that $H_i(\theta_j)=H_j(\theta_i)$. But as
we remarked above, the vanishing of $\oplus
R_{jk}^lh_{il}$ is equivalent to $h_{ik}W^k_j=h_{jk}W^k_i$. Thus
$h_{ij}$ satisfies the Helmholtz conditions.

Conversely, suppose that $h_{ij}$ satisfies the Helmholtz conditions.
Since $V_k(h_{ij})=V_j(h_{ik})$ there are
locally-defined functions $\bar{\theta}_i$, determined up to the
addition of arbitrary functions of the $x^i$ alone, such that
$h_{ij}=V_j(\bar{\theta}_i)$; and
$V_j(\bar{\theta}_i)=V_i(\bar{\theta}_j)$. We next show that for
any choice of the $\bar{\theta}_i$, the functions
$H_i(\bar{\theta}_j)-H_j(\bar{\theta}_i)$ are independent of the
$y^k$. Now
\[
\vf{y^k}\big(H_i(\bar{\theta}_j)\big)=H_i(h_{jk})-\conn lik h_{jl}.
\]
It is a simple and well-known consequence of the assumptions
that $(\nabla h)_{ij}=0$ and $V_k(h_{ij})=V_j(h_{ik})$ that
\[
H_i(h_{jk})-\conn likh_{jl}=
H_j(h_{ik})-\conn ljkh_{il},
\]
whence $H_i(\bar{\theta}_j)-H_j(\bar{\theta}_i)$ is independent of the
$y^k$. Thus
\[
\chi=\big(H_i(\bar{\theta}_j)-H_j(\bar{\theta}_i)\big)dx^i\wedge dx^j
\]
is a basic 2-form. We show that $\chi$ is closed. In
computing $d\chi$ we may replace the partial derivative with respect to
$x^k$ with $H_k$. We have
$\oplus H_k\big(H_i(\bar{\theta}_j)-H_j(\bar{\theta}_i)\big)
=\oplus[H_j,H_k](\bar{\theta}_i)=-\oplus R^l_{jk}h_{il}$.
But this vanishes if $h_{ik}W^k_j=h_{jk}W^k_i$. So $\chi$ is
closed, and hence (locally) exact. If now $\chi=d\psi$ with
$\psi=\psi_idx^i$, and $\theta_i=\bar{\theta}_i-\psi_i$, then
\[
\big(H_i(\theta_j)-H_j(\theta_i)\big)dx^i\wedge dx^j=\chi-d\psi=0.
\]
Set $\theta=\theta_i dx^i$.  Then $V_i(\theta_j)=h_{ij}$,
$d_J\theta=0$ and $d_H\theta=0$; moreover
$\lie{\Delta}\theta=h_{ij}y^j dx^i=0$; so $\theta$ satisfies the
differential conditions of Theorem~\ref{bm}.
\end{proof}

The condition on the rank of $d\theta$ gives the following corollary.

\begin{cor}
The projective class of sprays containing $\Gamma$
is the geodesic class of a pseudo-Finsler function if and only
if there is a semi-basic 1-form $\theta$ on $\TMO$ which
satisfies the differential conditions of Theorem~\ref{bm}, and
in addition $\rk(d\theta)=2n-2$.
\end{cor}

\begin{proof}
Let $\{dx^i,\phi^i\}$ be the local basis of 1-forms dual to the
local basis of vector fields $\{H_i,V_i\}$ corresponding to the
horizontal distribution determined by $\Gamma$. Then
\[
d\theta=H_i(\theta_j)dx^i\wedge dx^j-V_i(\theta_j)dx^i\wedge \phi^j
=-h_{ij}dx^i\wedge\phi^j.
\]
It follows from the fact that $h_{ij}y^j=0$ that $i_\Gamma
d\theta=i_\Delta d\theta=0$; thus in general $\rk(d\theta)\leq 2n-2$,
and $\rk(d\theta)=2n-2$ if and only if $h_{ij}$ is quasi-regular.
\end{proof}

The condition $i_\Gamma\theta>0$ now comes into its own in
ensuring that the pseudo-Finsler function is actually a Finsler
function:\ $\theta$ (if it exists with the given properties) is
the Hilbert 1-form, and $i_\Gamma\theta=F$, so this condition,
together with the rank condition on $d\theta$, say that there is
a positive pseudo-Finsler function. But it can be shown that a
pseudo-Finsler function which takes only positive values is a
Finsler function, a result originally due to Lovas \cite{Lovas}
which is quoted in \cite{BM}.

It is worth remarking, with reference to the relation between
Theorem \ref{bm} and Theorem \ref{old} below, that if one adds
to $\theta$ the pull-back of any closed 1-form on $M$ then
$d\theta$ is unchanged; and this operation corresponds exactly
to adding a total derivative to $F$. So in a sense the
inequality condition in Theorem \ref{bm} requires that there
must be, among all of the pseudo-Finsler functions with a given
Hilbert 2-form, determined up to the addition of a total
derivative, one (at least) which is everywhere positive. The
result of the analysis leading to Theorem~1 in \cite{mult}
suggests however that to expect this positivity condition to
hold globally over $M$ is somewhat ambitious.

\section{Formulations in terms of 2-forms}\label{2form}

Let $\Gamma$ be a (semi-)spray and $\{dx^i,\phi^i\}$ the local
basis of 1-forms corresponding to its horizontal distribution. A
symmetric tensor $h=h_{ij}(y) dx^i\otimes dx^j$ can always
be lifted to a 2-form $\omega=h_{ij}(y) dx^i\wedge\phi^j$ on
$\TMO$. This procedure was called the K\"{a}hler lift of $h$ in
\cite{MCSII}, since $\omega$ is clearly a generalization of the
K\"{a}hler form of a Riemannian metric.

Recall that for a given projective class of sprays we denote by $\D$ the
distribution on $\TMO$ spanned by $\Delta$ and any spray of the class;
it is involutive.

\begin{lem}
Suppose given a projective class of sprays, and a symmetric
tensor $h_{ij}$ such that $h_{ij}y^j=0$. Let $\Gamma$ be any
spray of the class, and $\omega=h_{ij}dx^i\wedge\phi^j$ the
corresponding K\"{a}hler lift of $h$. Then $\omega$ is a
concomitant of the class, that is, it is the same whichever
spray in the class is used to define it. Moreover, the
characteristic distribution of any such 2-form $\omega$ contains
the distribution $\D$ defined by the class.
\end{lem}

\begin{proof}
Any other member of the projective class is of the form
$\tilde{\Gamma}=\Gamma-2P\Delta$, where $P$ is a
positively-homogeneous function on $\TMO$. For the local basis
$\{dx^i,\tilde{\phi}^i\}$ corresponding to $\tilde{\Gamma}$ we have
\[
\tilde{\phi}^i=\phi^i+P dx^i+y^iV_j(P)dx^j,
\]
from which the first result readily follows.  Clearly
$i_\Gamma\omega=i_\Delta\omega=0$, as a consequence of the fact that
$h_{ij}y^j=0$.
\end{proof}

\begin{thm}\label{old}
Let $\Gamma$ be a spray, and let $\omega$ be a 2-form on $\TMO$ such
that
\begin{enumerate}
\item the characteristic distribution of $\omega$ contains $\D$, the
distribution spanned by the projective class of $\Gamma$;
\item $\lie{\Gamma}\omega=0$;
\item for any pair of vertical vector fields $V_1,V_2$,
$\omega(V_1,V_2)=0$;
\item for any horizontal vector field $H$ and any pair of vertical
vector fields $V_1,V_2$, $d\omega(H,V_1,V_2)=0$.
\end{enumerate}
Then over any coordinate neighbourhood $U\subset M$,
$\omega=h_{ij}dx^i\wedge \phi^j$ where $h_{ij}$ satisfies the
Helmholtz conditions. Conversely, if $h_{ij}$
satisfies the Helmholtz conditions then for any
spray $\Gamma$ in the projective class the 2-form
$\omega=h_{ij}dx^i\wedge\phi^j$ on $\tau^{-1}U$ has the
foregoing properties.
\end{thm}

Assumptions 3 and 4 may be stated as follows:\ for every $(x,y)\in
\TMO$ the vertical subspace of $T_{(x,y)}\TMO$ is isotropic for
$\omega$ and $i_Hd\omega$.

\begin{proof}
We may express $\omega$ in terms of the basis $\{dx^i,\phi^j\}$
defined by $\Gamma$.  It has no term in $\phi^i\wedge\phi^j$ because
of assumption~3.  Thus we may write
\[
\omega=a_{ij}dx^i\wedge dx^j+h_{ij}dx^i\wedge \phi^j
\]
where $a_{ji}=-a_{ij}$. A straightforward calculation yields
\[
\lie{\Gamma}\omega=
(\Gamma(a_{ij})-2a_{ik}\Gamma^k_j-h_{ik}R^k_j)dx^i\wedge dx^j
+((\nabla h)_{ij}+2a_{ij})dx^i\wedge\phi^j+h_{ij}\phi^i\wedge\phi^j.
\]
Since this must vanish, it follows (working from right to left)
that $h_{ij}$ is symmetric; that $(\nabla h)_{ij}=a_{ij}=0$
because one is symmetric, the other skew; and that $h_{ik}R^k_j$
is symmetric in $i$ and $j$. In particular,
$\omega=h_{ij}dx^i\wedge \phi^j$; it then follows from the first
assumption that $h_{ij}y^j=0$. Now
\[
d\omega=V_k(h_{ij})dx^i\wedge\phi^j\wedge\phi^k\pmod{dx^i\wedge
dx^j},
\]
so that
\[
d\omega(H_i,V_j,V_k)=V_k(h_{ij})-V_j(h_{ik})=0.
\]
Thus the coefficients $h_{ij}$ satisfy the Helmholtz conditions.

The converse is straightforward.
\end{proof}

\begin{cor}\label{oldclosed}
If a 2-form $\omega$ has the properties stated in Theorem~\ref{old}
then
\begin{enumerate}
\item $\omega(\hlift{X},\hlift{Y})=0$ and
$\omega(\hlift{X},\vlift{Y})=\omega(\hlift{Y},\vlift{X})$ for
any vector fields $X,Y$ on $M$;
\item $\omega$ is closed;
\item $\lie{Z}\omega=0$ for any vector field $Z$ in $\D$.
\end{enumerate}
\end{cor}

\begin{proof}
1. These follow from the explicit form of $\omega$.\\[5pt]
2. A straightforward calculation gives
\[
d\omega=\onehalf h_{il}R^l_{jk} dx^i\wedge dx^j\wedge dx^k+
\big(H_i(h_{jk})+h_{il}\conn ljk\big) dx^i\wedge dx^j\wedge\phi^k.
\]
The first term vanishes because $\oplus h_{il}R^l_{jk}=0$, the second
because $H_i(h_{jk})+h_{il}\conn ljk$ is symmetric in $i$ and $j$, as
we established in the proof of Theorem~\ref{bmmult}.\\[5pt]
3. For any $Z\in\D$, $\lie{Z}\omega=d(i_Z\omega)+i_Zd\omega=0$.
\end{proof}

\begin{cor}\label{oldcor}
A projective class of sprays is the geodesic class of a
locally-defined pseudo-Finsler function (that is, one defined over a
coordinate neighbourhood $U\subset M$) if and only if there is
a 2-form $\omega$ on $\tau^{-1}U$ with the properties stated in
Theorem~\ref{old}, such that the characteristic distribution of
$\omega$ is precisely the distribution $\D$ spanned by $\Delta$
and any spray of the class.
\end{cor}

We next describe how the positive quasi-definiteness condition
on $h$ may be specified as a condition on $\omega$. For any chosen
$\Gamma$ of the class, for $x\in M$ and $y\in\TxMO$ we define a
quadratic form $q_{(x,y)}$ on $T_xM$ by
$q_{(x,y)}(v)=\omega_{(x,y)}(\hlift{v},\vlift{v})$. Notice that
$q_{(x,y)}(y)=\omega_{(x,y)}(\Gamma,\Delta)=0$. This definition
may appear to depend on a choice of $\Gamma$ from the projective
class. The two-dimensional subspaces of $T_y\TMO$ of the form
$\langle\hlift{v},\vlift{v}\rangle$ are well-defined for a given
$\Gamma$, but change if $\Gamma$ is changed to a different
member of the projective class. But if we change $\Gamma$ to
$\Gamma-2P\Delta$ then $\hlift{v}$ changes to
$\hlift{v}-P\vlift{v}-\vlift{v}(P)\Delta_{(x,y)}$, and this
makes no difference to the value of
$\omega_{(x,y)}(\hlift{v},\vlift{v})$. So the quadratic form
$q_{(x,y)}$ is in fact a concomitant of the class.

\begin{cor}\label{oldfins}
If the quadratic form $q$ is positive quasi-definite everywhere
on $\TMO$ then in a neighbourhood of any point in $M$ there is a
local Finsler function of which the projective class
of sprays is the geodesic class.
\end{cor}

\begin{proof}
This follows directly from the explicit form of $\omega$.
\end{proof}

Putting these local results together with Theorem \ref{globmult}
we obtain the following global theorem.

\begin{thm}
If $F$ is a (global) Finsler function on $\TMO$ then its Hilbert
2-form $\omega=d\theta$ satisfies the conditions of
Theorem~\ref{old} for the sprays of its geodesic class, and in
addition the corresponding quadratic form $q$ is positive
quasi-definite everywhere. Conversely, suppose given a
projective class of sprays on $\TMO$. If there is a 2-form
$\omega$ on $\TMO$ which everywhere satisfies the conditions of
Theorem~\ref{old} and whose corresponding quadratic form $q$ is
everywhere positive quasi-definite, and if $H^2(M)=0$, then the
projective class is the geodesic class of a global
pseudo-Finsler function, and of a local Finsler function over a
neighbourhood of any point of $M$.
\end{thm}

We can illustrate the role of the cohomology condition in this
theorem, in a way a little different from the way it appears in
the proof of Theorem \ref{globmult} in \cite{mult}, by examining
the obstructions to the existence of a global semi-basic 1-form
$\theta$ whose exterior derivative is the closed 2-form
$\omega$.

We first prove the local version of the result. We shall make
use of the obvious fact that a form (of any degree) on $\TMO$
which is semi-basic and closed is basic (and closed).

\begin{lem}
Let $\omega$ be a 2-form on $\TMO$ which is closed and which
vanishes when both of its arguments are vertical. Then for any
contractible coordinate neighbourhood $U\subset M$ there is a
semi-basic 1-form $\theta$ on $\tau^{-1}U$ such that
$\omega=d\theta$.
\end{lem}

\begin{proof}
Set $\omega=a_{ij}dx^i\wedge dx^j+b_{ij}dx^i\wedge dy^j$. The
$dx\wedge dy\wedge dy$ term in $d\omega$ is
$V_k(b_{ij})dx^i\wedge dy^j\wedge dy^k$.
This must vanish, whence $V_k(b_{ij})=V_j(b_{ik})$.
Assuming that $\dim M\geq 3$ it follows that there are functions
$b_i(x,y)$ on $\tau^{-1}U$ such that
$b_{ij}(x,y)=V_j(b_i)(x,y)$. Then
\[
\omega+d(b_i dx^i)=\left(a_{ij}+\fpd{b_j}{x^i}\right)dx^i\wedge dx^j.
\]
The right-hand side is semi-basic and closed, so basic and
closed, so there is a 1-form $\psi$ on $U$ such that
$\omega+d(b_i dx^i)=d\psi$. Thus $\omega=d\theta$ with
$\theta=\psi-b_idx^i$, which is semi-basic.
\end{proof}

To derive the global theorem we shall need the following concepts and results.

An open covering $\mathfrak{U}=\{U_\lambda:\lambda\in\Lambda\}$
of $M$ which has the property that every $U_\lambda$, and every
non-empty intersection of finitely many of the $U_\lambda$, is
contractible is known as a good covering. It can be shown (see
\cite{mult}) that every manifold over which is defined a spray
admits good open coverings by coordinate patches.

The \v{C}ech cohomology group
$\check{H}^2(\mathfrak{U},\mathcal{R})$ of a good open covering
$\mathfrak{U}$ of $M$ is isomorphic to the de Rham cohomology
group $H^2(M)$. In particular, if $H^2(M)=0$ then
$\check{H}^2(\mathfrak{U},\mathcal{R})=0$; it is this form of
the assumption that we shall actually use in the proof.

Suppose that for a given good open covering $\mathfrak{U}$ of
$M$ by coordinate neighbourhoods, for each
$\lambda,\mu\in\Lambda$ for which $U_\lambda\cap U_\mu$ is
non-empty there is defined on $U_\lambda\cap U_\mu$ a function
$\phi_{\lambda\mu}$, and that these functions satisfy the
cocycle condition
$\phi_{\mu\nu}-\phi_{\lambda\nu}+\phi_{\lambda\mu}=0$ on
$U_\lambda \cap U_\mu \cap U_\nu$ (assuming it to be non-empty).
Then there is a locally finite refinement
$\mathfrak{V}=\{V_\alpha:\alpha\in A\}$ of $\mathfrak{U}$, and
for each $\alpha$ a function $\psi_\alpha$ defined on
$V_\alpha$, such that on $V_\alpha\cap V_\beta$ (assuming it to
be non-empty) $\phi_{\alpha\beta}=\psi_\alpha-\psi_\beta$, where
$\phi_{\alpha\beta}$ is defined from some $\phi_{\lambda\mu}$ by
restriction. This result, which is proved using a partition of
unity argument in \cite{mult}, is a particular case of the fact
that \v{C}ech cohomology is a sheaf cohomology theory (see
\cite{Warner}).

\begin{thm}
Let $\omega$ be a 2-form on $\TMO$ which is closed and which
vanishes when both of its arguments are vertical. Suppose that
$H^2(M)=0$. Then there is a semi-basic 1-form $\theta$ on $\TMO$
such that $d\theta=\omega$.
\end{thm}

\begin{proof}
Let $\mathfrak{U}$ be a good open covering of $M$ by coordinate
neighbourhoods. On each $U_\lambda$ there is a semi-basic 1-form
$\theta_\lambda$ such that $\omega=d\theta_\lambda$. On
$U_\lambda\cap U_\mu$, $d(\theta_\lambda-\theta_\mu)=0$; that
is, $\theta_\lambda-\theta_\mu$ is semi-basic and closed, so
there is a function $\phi_{\lambda\mu}$ on $U_\lambda\cap U_\mu$
such that $\theta_\lambda-\theta_\mu=d\phi_{\lambda\mu}$. On
$U_\lambda \cap U_\mu \cap U_\nu$,
$d(\phi_{\mu\nu}-\phi_{\lambda\nu}+\phi_{\lambda\mu})=0$, so
$\phi_{\mu\nu}-\phi_{\lambda\nu}+\phi_{\lambda\mu}$ is a
constant, say $k_{\lambda\mu\nu}$. For any four members
$U_\kappa,U_\lambda,U_\mu,U_\nu$ of $\mathfrak{U}$ whose
intersections in threes are non-empty
\[
 k_{\lambda\mu\nu}-k_{\kappa\mu\nu}+k_{\kappa\lambda\nu}
-k_{\kappa\lambda\mu}=0.
\]
That is to say, $k$ is a 2-cocycle in the \v{C}ech cochain complex
for the covering $\mathfrak{U}$ with values in the constant
sheaf $M\times\R$. Under the assumption that
$H^2(M)=\check{H}^2(\mathfrak{U},\mathcal{R})=0$, it must be a
coboundary. Thus we can modify each $\phi_{\lambda\mu}$ by the
addition of a constant, so that (after modification)
$\phi_{\mu\nu}-\phi_{\lambda\nu}+\phi_{\lambda\mu}=0$. There is
thus a refinement $\mathfrak{V}=\{V_\alpha:\alpha\in A\}$ of
$\mathfrak{U}$, and for each $\alpha$ a function $\psi_\alpha$
defined on $V_\alpha$, such that on $V_\alpha\cap V_\beta$
(assuming it to be non-empty)
$\phi_{\alpha\beta}=\psi_\alpha-\psi_\beta$. But then
$\theta_\alpha-d\psi_\alpha=\theta_\beta-d\psi_\beta$ on
$V_\alpha\cap V_\beta$. So if we set
$\theta=\theta_\alpha-d\psi_\alpha$ on $V_\alpha$, $\theta$ is a
well-defined semi-basic 1-form on $\TMO$ such that
$d\theta=\omega$.
\end{proof}

We have shown that when $H^2(M)=0$ there is a globally defined
semi-basic 1-form $\theta$ such that $d\theta=\omega$. In virtue
of the other conditions on $\omega$, this 1-form must satisfy
the differential conditions of the theorem of Bucataru and
Muzsnay.

\section{Almost Grassmann structures}

We now make a detour to discuss another approach to the
construction of a 2-form indicating Finsler metrizability, which
gives a new geometrical interpretation of the vertical
subspaces, on the one hand, and the two-dimensional subspaces of
the form $\langle\hlift{v},\vlift{v}\rangle$, on the other,
which play an important role in the conditions for the
existence of a Finsler function discussed in the previous
section. This approach necessitates the use of an almost
Grassmann structure~\cite{AG}.

Formally, an almost Grassmann structure on a manifold $N$ of
dimension $pq$, $p \geq 2$, $q \geq 2$, may be regarded as a
Cartan geometry modelled on the Grassmannian of $p$-dimensional
subspaces of $\R^{p+q}$~\cite{Sharpe}. One way to define such a
structure is by specifying a class of local bases of 1-forms
$\{\theta^i_\a\}$, any two such local bases of the class being
related by a formula $\hat{\theta}^i_\a = B^i_j A^\b_\a
\theta^j_\b$ where $(A^\a_\b)$ and $(B^i_j)$ are local
matrix-valued functions, respectively $p\times p$ and $q\times
q$, both non-singular.

Given an almost Grassmann structure, we denote the local basis
of vector fields dual to a local basis of 1-forms $\{\theta^i_\a \}$
in the structure by $\{ E^\a_i \}$, so that any
vector $v\in T_x N$ may be written as $v_\a^i E^\a_i(x)$. Of
special interest are those vectors $v$ for which the coefficient
matrix $(v_\a^i)$ has rank 1; the set of such $v$ forms a cone
in $T_x N$ called the \Se\ cone. That is to say, the \Se\ cone
at $x \in N$ consists of those elements of $T_x N$ that can be
expressed in the form $s_\a t^i E^\a_i(x)$ with respect to one,
and hence any, basis $\{ E^\a_i \}$ defined by the structure,
where $(s_\a)\in \R^p$ and $(t^i) \in \R^q$. For fixed non-zero
$(t^i)$, as $(s_\a)$ varies over $\R^p$ we obtain a
$p$-dimensional subspace of $T_x N$ contained in the \Se\ cone;
we call it a $p$-dimensional plane generator of the \Se\ cone.
The $p$-dimensional plane generators of \Se\ cones are
parametrized by the points of the projective space $\P^{q-1}$.
Similarly, on fixing non-zero $(s_\a)$, as $(t^i)$ varies over
$\R^q$ we obtain a $q$-dimensional plane generator of the \Se\
cone.

There is an almost Grassmann structure of type $(2,n)$
associated with each projective class of sprays on the
$2n$-dimensional manifold $\TMO$. This structure is not,
however, defined on $\TMO$ itself, but on a related
$2n$-dimensional bundle $\traco\to M$ obtained from a vector
bundle $\trac\to M$ by deleting the zero section.

We may construct $\trac$ using a technique described
in~\cite{Japan}. We let $\VM$ be the manifold of equivalence
classes $[\pm\theta]$ of non-zero volume elements $\theta\in
\bigwedge^n T^* M$, and let $\nu : \VM \to M$ be given by
$\nu([\pm\theta]) = x$ where $\theta,-\theta \in \bigwedge^n
T^*_x M$. Given coordinates $x^i$ on $M$, define the map $v$ by
\[
\theta = v(\theta) \left( dx^1 \wedge \cdots \wedge dx^n \right)_x
\]
and let $x^0 = |v|^{1/(n+1)}$ be a fibre coordinate on $\nu$. In
this way $\nu : \VM \to M$ becomes a principal $\R_+$ bundle
with fundamental vector field $\U = x^0 \partial / \partial
x^0$. Now consider the tangent bundle $T\VM \to \VM$ and the
vector fields
\[
\vlift{\U} = x^0 \fpd{}{y^0} , \quad \Ut = \clift{\U} - \Dt
= x^0 \fpd{}{x^0} - y^i \fpd{}{y^i}
\]
where $\Dt$ is the dilation field on $T\VM$. The distribution
spanned by these two vector fields is integrable, and the
quotient is a manifold $\trac$ which does not project to $\VM$
but does define a vector bundle over $M$. The fibre coordinates
$(u^i)$ on the new bundle are defined in terms of the fibre
coordinates $(y^i)$ of $TM$ by $u^i = x^0 y^i$; the quotient
manifold may be thought of as the tensor product of the ordinary
tangent bundle with the bundle of scalar densities of weight
$1/(n + 1)$.

The construction of the almost Grassmann structure
may also be found in~\cite{Japan}. For any spray
\[
y^i \fpd{}{x^i} - 2 \Gamma^i \fpd{}{y^i}
\]
on $\TMO$ there is a well-defined horizontal distribution on
$\traco$, spanned locally by the vector fields
\[
\K_i = \fpd{}{x^i} -
\left( \Gamma^j_i - \frac{1}{n+1} u^j \Gamma_i \right) \fpd{}{u^j} ,
\]
where $u^i$ are the natural fibre coordinates on $\traco$ and
\[
\Gamma^j_i = \fpd{\Gamma^j}{y^i}, \quad
\Gamma_i=\fpd{\Gamma^k_k}{y^i} .
\]
If two sprays are related by a projective transformation with
function $P$, the vector fields are modified according to
the rule
\[
\K_i \mapsto \K_i -P\fpd{}{u^i}.
\]
We shall write, for $v\in T_xM$,
\[
\hlift{v}=v^i\K_i,\quad
\vlift{v}=v^i\vf{u^i}.
\]

Now suppose given a projective class of sprays. Choose a
particular spray in the class; from the remarks above we see
that in a coordinate patch with coordinates $(x^i,u^i)$ the
1-forms
\[
\theta_1^i = dx^i, \quad \theta_2^i =
du^i + \left( \Gamma^i_j - \frac{1}{n+1} u^i \Gamma_j \right)dx^j
\]
transform as
\[
\hat{\theta}^i_1 = J^i_j \theta^j_1, \quad \hat{\theta}^i_2 =
 |J|^{-1/(n+1)} J^i_j \theta^j_2
\]
under a coordinate transformation, where $J^i_j$ is the Jacobian
matrix of the transformation on $M$ and $|J|$ is its
determinant, and as
\[
\hat{\theta}^i_1 = \theta^i_1, \quad \hat{\theta}^i_2
= \theta^i_2 + P\theta^i_1
\]
under a projective transformation. It follows that the set of
locally-defined 1-forms $\{A^i_jA^\b_\a\theta^j_\b\}$, with
$\a,\b=1,2$, $i,j=1,2,\ldots,n$, with $(A^i_j)$, $(A^\b_\a)$
arbitrary local non-singular-matrix-valued functions, of size
$n\times n$ and $2\times 2$ respectively, is defined
independently of the choice of coordinates and of the choice of
spray within the projective class. These 1-forms
therefore determine an almost Grassmann structure of type
$(2,n)$ on $\traco$.

The Segre cone at a point $(x,u)$ of $\traco$ consists of
vectors of the form $a\hlift{v}+b\vlift{v}$ for $a,b\in\R$ and
$v\in T_xM$. The $n$-dimensional plane generators of the Segre
cone are obtained by fixing $a$ and $b$ and letting $v$ range
over $T_xM$; they consist of
the horizontal subspace with respect to each spray of the
projective class together with the vertical subspace. For the
two-dimensional plane generators we fix $v$ and allow the
coefficients to vary over $\R^2$. Notice that $\D_{(x,u)}$, where
$\D$ is the involutive two-dimensional distribution spanned by
$\Delta$ and any spray of the class, is a two-dimensional plane
generator of the Segre cone at $(x,u)$.

Each Finsler geometry on $\TMO$ determines a projective class of
sprays, and therefore determines an almost Grassmann structure
on $\traco$. We may determine a relationship between the two
structures using the fact that if $\omega$ is a closed form on
$\TMO$ satisfying $i_\Delta\omega = 0$ then its pull-back
$(\nu_*)^*\omega$ by $\nu_* : \TVO \to \TMO$ is projectable to a
form on $\traco$, and apply this to the Hilbert 2-form.

\begin{thm} {\em \cite{Japan}}
To each Finsler function $F$ on $\TMO$ there is associated a
closed 2-form $\varpi$ on $\traco$, such that the
characteristic distribution of $\varpi$ is the two-dimensional
distribution $\D$ corresponding to the geodesic sprays of $F$,
and such that the $n$-dimensional plane generators of the \Se\
cones are isotropic with respect to $\varpi$.

Conversely, suppose given a projective class of
sprays on $\TMO$ and corresponding almost Grassmann structure.
If there is a 2-form $\varpi$ on $\traco$ such that
\begin{enumerate}
\item the $n$-dimensional plane generators of the \Se\ cones
are isotropic with respect to $\varpi$;
\item the characteristic distribution of $\varpi$ is $\D$;
\item $\varpi$ is closed
\end{enumerate}
then the projective class is the geodesic class of a locally-defined
pseudo-Finsler function.
\end{thm}

\begin{proof}
In fact $\varpi$ must take the form $h_{ij}dx^i\wedge
\theta_2^j$ with $h_{ij}$ the Hessian, or putative Hessian, of
$F$, much as in Theorem \ref{old}:\ see \cite{Japan} Theorems~4
and~6 for the details.
\end{proof}

We can now further refine this result. Let $\alpha$ be a
2-covector on some vector space $V$ of dimension at least two,
and $W$ a two-dimensional subspace of $V$. Then either
$\alpha|_W\equiv 0$, or $\alpha(w_1,w_2)=0$ for $w_1,w_2\in W$
only if $w_1$ and $w_2$ are linearly dependent. In the former
case we say that $\alpha$ vanishes on $W$.

\begin{cor}
If the 2-form $\varpi$ vanishes on a 2-plane
generator of Segre cones only if it is a generator determined by
$\D$ then either $h$ or $-h$ is positive quasi-definite, and there
is a local Finsler function whose geodesic class is the given
projective class.
\end{cor}

\begin{proof}
At any $(x,u)\in\traco$ we have, for $v\in T_xM$,
$\varpi(\hlift{v},\vlift{v})=h_{ij}v^iv^j$. This cannot
vanish unless $v$ is a scalar multiple of $u$. Thus at each
point of $\traco$ $h$ is either positive or negative
quasi-definite. By continuity either $h$ or $-h$ must be positive
quasi-definite everywhere.
\end{proof}

\section{Path space}

We take up the argument from where we left it in Section \ref{2form}.

We now assume that we can quotient out by the foliation on
$\TMO$ defined by the involutive distribution $\mathcal{D}$,
that is, that there is a $(2n-2)$-dimensional manifold
$P_{\mathcal{D}}$ --- {\em path space} --- such that
$\pi:\TMO\to P_{\mathcal{D}}$ is a fibration whose fibres are
the leaves of the foliation.  So we have a double fibration
\[
M\stackrel{\tau}{\longleftarrow} \TMO
\stackrel{\pi}{\longrightarrow} P_{\mathcal{D}}.
\]
(It has to be admitted that in general there is no reason for the path space of a projective class of sprays to be a smooth manifold. For the geodesic class of sprays of a Riemmannian metric, for example, two well-known cases where the path space can be given the structure of a smooth manifold are the cases where the Riemannian manifold is either a Hadamard manifold (i.e. a complete simply connected Riemannian manifold of non-positive curvature) or a manifold with closed geodesics of the same length. These two cases are discussed in detail in e.g.\ Ferrand \cite{Ferrand} and Besse \cite{Besse}, respectively.)

For any $x\in M$, denote by $\hat{x}$ the submanifold $\pi(\TxMO)$ of
$P_{\mathcal{D}}$, that is, the image under $\pi$ of the fibre of
$\TMO$ over $x$:\ it is the submanifold consisting of all paths through
$x$.  It is of dimension $n-1$, because $\Delta$ is vertical.

The following theorem is our version of Theorem~4.1 of \cite{AP}.

\begin{thm}\label{AP}
A projective class of sprays is the geodesic class of a
pseudo-Finsler function if and only if there is a symplectic 2-form
$\Omega$ on $P_{\mathcal{D}}$ such that $\hat{x}$ is a Lagrangian
submanifold of $P_{\mathcal{D}}$ with respect to $\Omega$, for every
$x\in M$.
\end{thm}

\begin{proof}
Suppose that the projective class of sprays is
derivable from a pseudo-Finsler function. Let
$\omega=h_{ij}dx^i\wedge\phi^j$ be the Hilbert 2-form on $\TMO$.
It satisfies the conditions of Theorem~\ref{old}, is closed, has
$\D$ for its characteristic distribution, and satisfies
$\lie{Z}\omega=0$ for every vector field $Z$ in $\D$. It
therefore passes to the quotient, that is, there is a 2-form
$\Omega$ on $P_{\mathcal{D}}$ such that $\pi^*\Omega=\omega$.
Then $\Omega$ is non-singular. Moreover $\pi^*d\Omega=0$; but
$\pi$ is surjective, so $d\Omega=0$. Thus $\Omega$ is
symplectic. Let $p\in\hat{x}$ and $\xi,\eta\in T_p\hat{x}$. Then
there is $y\in\TxMO$, and $v,w\in T_y\TxMO$ (i.e.\ vertical vectors
at $y$) such that $p=\pi(y)$, $\xi=\pi_*v$, $\eta=\pi_*w$, and
\[
\Omega_p(\xi,\eta)=\Omega_p(\pi_*v,\pi_*w)
=\pi^*\Omega_p(v,w)=\omega_y(v,w)=0.
\]
Conversely, suppose that there is a 2-form $\Omega$ on
$P_{\mathcal{D}}$ with the stated properties.  Set
$\omega=\pi^*\Omega$.  Then $d\omega=0$, and the characteristic
distribution of $\omega$ is $\mathcal{D}$.  Evidently
$\lie{\Gamma}\omega=0$ for any spray in the class.  Let $x\in M$,
$y\in\TxMO$, and $v,w\in T_y\TxMO$ (i.e.\ any vertical vectors at $y$).
Then
\[
\omega_y(v,w)=\pi^*\Omega_p(v,w)=\Omega_p(\pi_*v,\pi_*w)=0
\]
because $\pi_*v,\pi_*w\in T_p\hat{x}$.  Now apply
Corollary~\ref{oldcor}.
\end{proof}

When $\dim M=2$ the dimension of $P_\D$ is also 2, so in this
case there is essentially no condition in Theorem~\ref{AP},
because every 2-form is closed, and $\hat{x}$ is 1-dimensional.
That is, every volume form (nowhere vanishing 2-form) on $P_\D$
satisfies the conditions of the theorem. Moreover, we see that
the freedom in choice of Hessians of pseudo-Finsler functions is
the same as the freedom in choice of volume forms in 2
dimensions, that is, multiplication by a function on $P_\D$.
These observations give another interpretation of the results on
the two-dimensional case in \cite{mult}.

\section{Jacobi fields}

Roughly speaking, a point in path space $P_{\mathcal{D}}$ represents a
geodesic, and so a tangent vector to path space at a point in it is an
`infinitesimal connecting vector to a nearby geodesic', that is, a
Jacobi field along the initial geodesic.  This observation, when
tidied up, gives another interpretation of the requirement that
$\hat{x}$ is a Lagrangian submanifold of $P_{\mathcal{D}}$ with
respect to $\Omega$.

In order to discuss Jacobi fields we have to fix the parametrization,
that is, choose a specific spray $\Gamma$ from the projective
equivalence class.  However, since the argument to be presented below
leads merely to a reinterpretation of the conditions just mentioned,
which we know from the previous section to be defined for the whole
projective class, it clearly makes no difference which particular
spray from the class we choose to work with.

Let $t\mapsto \gamma(t)\in M$ be a geodesic, that is, a base
integral curve of $\Gamma$.  Then $t\mapsto
\bar{\gamma}(t)=(\gamma(t),\dot{\gamma}(t))$ is an integral curve of
$\Gamma$ in $\TMO$, and in coordinates
\[
\ddot{\gamma}^i(t)+2\Gamma^i(\gamma(t),\dot{\gamma}(t))=0.
\]
Let $Z$ be a vector field along $\bar{\gamma}$ such that
$\lie{\Gamma}Z=0$.  We set $\zeta=\tau_*Z$, a vector field along the
geodesic $\gamma$; then the condition $\lie{\Gamma}Z=0$ is equivalent
to $Z=\hlift{\zeta}+\vlift{(\nabla\zeta)}$ where
$\nabla^2\zeta^i+R_j^i\zeta^j=0$.  That is to say, $\zeta$ is a Jacobi
field along $\gamma$, and there is a 1-1 correspondence between Jacobi
fields along $\gamma$ and vector fields which are Lie transported
along $\bar{\gamma}$.  Evidently $\dot{\gamma}$ is a Jacobi field
along $\gamma$, corresponding to the restriction of $\Gamma$ to
$\bar{\gamma}$.  Moreover, $t\mapsto t\dot{\gamma}(t)$ is a Jacobi
field along $\gamma$, corresponding to the restriction to
$\bar{\gamma}$ of $t\Gamma+\Delta$.  These Jacobi fields in the
tangent direction of $\gamma$ may be regarded as trivial.  We denote
by $\J_\gamma$ the space of Jacobi fields along $\gamma$.  It is a
$2n$-dimensional real vector space.  We denote by $\J^0_\gamma$ the
quotient of $\J_\gamma$ by the two-dimensional subspace consisting of
the trivial Jacobi fields which lie in the direction tangent to
$\gamma$.

There is a leaf of the involutive distribution $\D$ containing
$\bar{\gamma}$: call it $\L_\gamma$. It consists of all points of
$\TMO$ of the form $(\gamma(t),e^s\dot{\gamma}(t))$ for
$(s,t)\in\R^2$ (assuming that $\gamma$ is defined on $\R$). The leaf
$\L_\gamma$ determines a point $p=\pi(\L_\gamma)\in P_\D$. Now let
$Z$ be a vector field defined over $\L_\gamma$ (but not tangent to
it; strictly speaking, $Z$ is a vector field along the injection
$\L_\gamma\to\TMO$). The Lie derivative of such a vector field $Z$ by
any vector field in $\D$ is well defined; and $Z$
projects to an element of $T_p P_\D$ if and only if every such Lie
derivative lies in $\D|_{\L_{\gamma}}$. That is to say, for every
vector field $Z$ on $\L_\gamma$ such that $\lie{\Delta}Z\in\D$ and
$\lie{\Gamma}Z\in\D$, $\pi_*Z$ is a well-defined element of $T_p
P_\D$; and every element of $T_p P_\D$ is of this form for some such
$Z$.

We shall show that there is an isomorphism of $\J_\gamma^0$ with
$T_pP_\D$.  We know that any $\zeta\in\J_\gamma$ lifts to a vector
field $Z$ along $\bar{\gamma}$ such that $\lie{\Gamma}Z=0$.  We shall
first show that such a vector field $Z$ can be extended to a vector
field (also denoted by $Z$) on $\L_\gamma$ such that
$\lie{\Delta}Z=\lie{\Gamma}Z=0$.

\begin{lem}\label{extend}
Let $t\mapsto Z(t)$ be a vector field along $\bar{\gamma}$ such that
$\lie{\Gamma}Z=0$. Then there is a unique vector field $(s,t)\mapsto
Z(s,t)$ on $\L_\gamma$ such that $\lie{\Delta}Z=\lie{\Gamma}Z=0$ and
$Z(t)=Z(0,t)$.
\end{lem}

\begin{proof}
Let $\delta_s$ be the 1-parameter group generated by $\Delta$
acting on $\L_\gamma$, so that for any $(x,y)\in\L_\gamma$,
$\delta_s(x,y)=(x,e^sy)$. Let $Z(t)$ be any vector field along
$\bar{\gamma}$ and set $Z(s,t)=\delta_{s*}Z(t)$. Then
$\lie{\Delta}Z=0$ and $Z(0,t)=Z(t)$; moreover $Z(s,t)$ is
uniquely determined by these properties. Now
$[\Delta,\Gamma]=\Gamma$, whence
$\lie{\Delta}\lie{\Gamma}Z=\lie{\Gamma}\lie{\Delta}Z+\lie{\Gamma}Z
=\lie{\Gamma}Z$. It follows that
$(\lie{\Gamma}Z)(s,t)=e^s\delta_{s*}(\lie{\Gamma}Z)(0,t)
=e^s\delta_{s*}(\lie{\Gamma}Z)(t)$. So if $(\lie{\Gamma}Z)(t)=0$
then $(\lie{\Gamma}Z)(s,t)=0$.
\end{proof}

We define a linear map $j:\J_\gamma\to T_pP_\D$ as follows. For
$\zeta\in\J_\gamma$ let $Z(t)$ be the corresponding vector field
along $\bar{\gamma}$, and $Z(s,t)$ the vector field on $\L_\gamma$
whose existence is guaranteed by the lemma. Then $\pi_*Z$ is a
well-defined element of $T_pP_\D$, and we set $\pi_*Z=j(\zeta)$. We
shall show that the kernel of $j$ is spanned by the trivial Jacobi
fields $\dot{\gamma}$ and $t\dot{\gamma}$, whence $j:\J^0_\gamma\to
T_pP_\D$ is an isomorphism by dimension.

\begin{prop}
The linear map $j:\J^0_\gamma\to T_pP_\D$ is an isomorphism.
\end{prop}

\begin{proof}
Let us denote by $\varphi$ the map $\R^2\to\L_\gamma$ given by
$\varphi(s,t)=(\gamma(t),e^s\dot{\gamma}(t))$. Then evidently
\[
\varphi_*\left(\vf{s}\right)=\Delta_{\varphi(s,t)}.
\]
Furthermore
\[
\varphi_*\left(e^s\vf{t}\right)=
e^s\left(\dot{\gamma}^i(t)\vf{x^i}
+e^s\ddot{\gamma}(t)^i\vf{y^i}\right).
\]
But $\ddot{\gamma}(t)^i=-2\Gamma^i(\gamma(t),\dot{\gamma}(t))$,
and $\Gamma^i$ is positively homogeneous of degree 2 in the
fibre coordinates; thus
\begin{align*}
\varphi_*\left(e^s\vf{t}\right)&=
e^s\dot{\gamma}^i(t)\vf{x^i}
-2e^{2s}\Gamma^i(\gamma(t),\dot{\gamma}(t))\vf{y^i}\\
&=e^s\dot{\gamma}^i(t)\vf{x^i}
-2\Gamma^i(\gamma(t),e^s\dot{\gamma}(t))\vf{y^i}\\
&=\Gamma_{\varphi(s,t)}.
\end{align*}
Thus we can use $s$ and $t$ as coordinates on $\L_\gamma$, with
\[
\Delta=\vf{s},\quad \Gamma=e^s\vf{t}.
\]
A vector field on $\L_\gamma$ which projects onto $0\in T_pP_\D$
takes the form
\[
Z(s,t)=\sigma(s,t)\vf{s}+\tau(s,t)\vf{t}.
\]
Then $\lie{\Delta}Z=0$ if and only if $\sigma$ and $\tau$ are
independent of $s$. Furthermore,
\[\lie{\Gamma}Z=e^s\left(\fpd{\sigma}{t}\vf{s}
+\left(\fpd{\tau}{t}-\sigma\right)\vf{t}\right),
\]
so that $\lie{\Gamma}Z=0$ if and only if $\sigma=a$ is constant and
$\tau(s,t)=at+b$ where $b$ is constant. Then
\[
Z(s,t)=a\Delta_{\varphi(s,t)}+(at+b)e^{-s}\Gamma_{\varphi(s,t)},
\]
and in particular
\[
Z(0,t)=a(t\Gamma_{\bar{\gamma}(t)}+\Delta_{\bar{\gamma}(t)})
+b\Gamma_{\bar{\gamma}(t)}.
\]
That is to say, $Z$ corresponds to a linear combination of trivial
Jacobi fields, and so the kernel of $j$ is spanned by the trivial Jacobi fields.
\end{proof}

Now let $\gamma$ be a geodesic of $\Gamma$ through $x\in M$,
with $\gamma(0)=x$ for convenience. Denote by $\J_{\gamma,0}$
the space of Jacobi fields along $\gamma$ which vanish at $x$,
and $\J_{\gamma,0}^0$ the quotient of $\J_{\gamma,0}$ by the
constant multiples of $t\dot{\gamma}(t)$. Then $j$ maps
$\J_{\gamma,0}^0$ onto $T_p\hat{x}$, and is an isomorphism.

Let $\omega$ be a 2-form on $\TMO$ such that
$\lie{\Gamma}\omega=0$ (Theorem~\ref{old} assumption 2). Let
$\zeta_1$ and $\zeta_2$ be Jacobi fields along $\gamma$, and
$Z_1$ and $Z_2$ the corresponding vector fields on $\L_\gamma$
as given in Lemma~\ref{extend}. Then since
$\lie{\Gamma}Z_1=\lie{\Gamma}Z_2=0$,
$\Gamma(\omega(Z_1,Z_2))=0$, that is to say, $\omega(Z_1,Z_2)$
is constant along every integral curve of $\Gamma$ in
$\L_\gamma$. Suppose further that for $x\in M$,
$\omega|_{\TxMO}=0$ (Theorem~\ref{old} assumption 3). If now
$\zeta_1(0)=\zeta_2(0)=0$, so that $Z_1(0)$ and $Z_2(0)$ are
vertical, then $\omega(Z_1,Z_2)=0$ on the ray $s\mapsto
e^s\vlift{\dot{\gamma}(0)}$ in $\L_\gamma$. But this is
transversal to $\Gamma$ in $\L_\gamma$, so $\omega(Z_1,Z_2)=0$
on $\L_\gamma$.

It is this property of $\omega$ which corresponds to the property that
the submanifolds $\hat{x}$ are Lagrangian in Theorem~\ref{AP} (when
$\omega$ satisfies the conditions of Theorem~\ref{old} and
Corollary~\ref{oldclosed}):\ note that $\pi(\L_\gamma)=p\in\hat{x}$,
and we have in effect shown that $\Omega_p$ (where $\Omega$ is the
projection of $\omega$) vanishes on any pair of vectors in
$T_p\hat{x}$.

We consider next the positive quasi-definiteness condition. Take
$x\in M$ and $y\in\TxMO$, and let $\gamma$ be the geodesic with
$\gamma(0)=x$ and $\dot{\gamma}(0)=y$. Let $v\in T_xM$:\ there
are unique Jacobi fields $\zeta_1(t),\zeta_2(t)$ along $\gamma$
such that
\[
\zeta_1(0)=v,\quad \nabla\zeta_1(0)=0;\qquad
\zeta_2(0)=0,\quad\nabla\zeta_2(0)=v.
\]
Let $Z_1,Z_2$ be the corresponding vector fields along
$\bar{\gamma}$ such that $\lie{\Gamma}Z_1=\lie{\Gamma}Z_2=0$.
Then the quadratic form $q_{(x,y)}$ on $T_xM$
defined in Section \ref{2form} is given by
\[
q_{(x,y)}(v)=\omega_{(x,y)}(\hlift{v},\vlift{v})
=\omega_{(x,y)}(Z_1,Z_2)=\Omega_p(\zeta_1,\zeta_2)
\]
where $p=\pi(x,y)$ and $\zeta_1,\zeta_2\in T_pP_{\mathcal{D}}$
are the elements determined by the Jacobi fields
$\zeta_1(t),\zeta_2(t)$ by means of $j$. We have the following version of
Corollary \ref{oldfins}.

\begin{cor}
A projective class of sprays is the geodesic class
of a local Finsler function if and only if there is a symplectic
2-form $\Omega$ on $P_{\mathcal{D}}$ such that $\hat{x}$ is a
Lagrangian submanifold of $P_{\mathcal{D}}$ with respect to
$\Omega$, for every $x\in M$, and moreover
$\Omega_p(\zeta_1,\zeta_2)>0$ for all non-zero
$\zeta_1,\zeta_2\in T_pP_{\mathcal{D}}$ of the special form
described above.
\end{cor}

\section{Totally-geodesic submanifolds}

In \'{A}lvarez Paiva's analysis of Hilbert's fourth problem
(\cite{AP}; see also \cite{hil4}) 2-planes in $\R^n$ play an
important role:\ see for example Theorem 4.5 of \cite{AP}, which
relates the positivity properties of an admissible 2-form
$\omega$ (our $\Omega$) to the pull-back of $\omega$ to each
two-dimensional submanifold of path space consisting of all lines
in a 2-plane. This can be generalized to the kind of situation
discussed here if we change 2-planes to two-dimensional
totally-geodesic submanifolds, as we now explain.

Let $N$ be a proper embedded submanifold of $M$.  We define a
submanifold $\hat{N}$ of $\TMO$, of twice the dimension, as follows:\
$\hat{N}=\{(x,y)\in\TMO:y\in T_xN\}$.  Thus $\hat{N}$ is $\TNO$
considered as a submanifold of $\TMO$.  Evidently $\Delta$ is tangent
to $\hat{N}$ (if $y\in T_xN$ then also $e^ty\in T_xN$).  Moreover, if
$v$ is any vector tangent to $N$ then $\vlift{v}$ is tangent to
$\hat{N}$, since $\vlift{v}_y$ is the tangent at $t=0$ to the curve
$t\mapsto y+tv$, and if $y\in T_xN$ and $v\in T_xN$ then $y+tv\in
T_xN$ for all $t$.

We say that the submanifold $N$ is {\it totally geodesic\/} with respect
to the spray $\Gamma$ if $\Gamma$ is tangent to $\hat{N}$.  Then
every geodesic $\gamma$ of $\Gamma$ in $M$ which starts at a point
$x=\gamma(0)$ of $N$ and is tangent to $N$ there (so that
$\dot{\gamma}(0)\in T_xN$) lies totally within $N$:\ it is the
projection of the integral curve of $\Gamma$ through
$(\gamma(0),\dot{\gamma}(0))$, which lies in $\hat{N}$.  Note that
since $\Delta$ is tangent to $\hat{N}$, if $\Gamma$ is tangent to
$\hat{N}$ so is any projectively-equivalent spray:\ that is, being
totally geodesic is a projective property (as it should be, since it
should be concerned with geodesic paths rather than parametrized
geodesics).

\begin{lem}
If $N$ is totally geodesic then if $v$ is any vector tangent to $N$,
$\hlift{v}$ is tangent to $\hat{N}$.
\end{lem}

\begin{proof}
We can find coordinates on $M$ such that $N$ is
given by $x^\alpha=0$, $\alpha=\dim{N}+1,\ldots,n$.  We use $a,b$ for
indices $1,\ldots,\dim{N}$.  Clearly $\hat{N}$ is given by
$x^\alpha=0$, $y^\alpha=0$.  With
\[
\Gamma=y^i\vf{x^i}-2\Gamma^i\vf{y^i},
\]
$N$ is totally geodesic if and only if
$\Gamma^\alpha(x^a,0,y^a,0)=0$. Now
\[
H_a=\vf{x^a}-\Gamma^i_a\vf{y^i}.
\]
But on $\hat{N}$
\[
\Gamma^\alpha_a(x^b,0,y^b,0)
=\fpd{\Gamma^\alpha}{y^a}(x^b,0,y^b,0)=0.
\]
Thus on $\hat{N}$
\[
H_a=\vf{x^a}-\Gamma^b_a\vf{y^b},
\]
which is tangent to $\hat{N}$.
\end{proof}

For convenience, when speaking of vector fields in relation to a
submanifold we shall use `on' to mean not just `defined on' but also
`tangent to'.

When $N$ is totally geodesic the space of vector fields on
$\hat{N}$ is spanned by the vector fields $\vlift{X}$,
$\hlift{X}$, where $X$ is any vector field on $N$; notice that
$\vlift{X}$ and $\hlift{X}$ coincide with $\Delta$ and $\Gamma$
where $y=X(x)$, that is, on the image of the corresponding
section. Now
\[
[\Gamma,\vlift{X}]=-\hlift{X}+\vlift{(\nabla X)},\quad
[\Gamma,\hlift{X}]=\hlift{(\nabla X)}+\vlift{\Phi(X)}.
\]
For a totally-geodesic submanifold these formulas make sense on
$\hat{N}$ with $X$ any vector field on $N$. Then since $\Gamma$,
$\vlift{X}$ and $\hlift{X}$ are all tangent to $\hat{N}$, so is
$\vlift{(\nabla X)}$, and so is $\vlift{\Phi(X)}$. Of course
$[\Delta,\vlift{X}]=-\vlift{X}$ and $[\Delta,\hlift{X}]=0$.

We now consider two-dimensional totally-geodesic submanifolds.  Any such
submanifold $N$ defines a two-dimensional submanifold $\bar{N}$ of
$P_\D$, whose points consist of geodesic paths in $N$; alternatively,
$\bar{N}=\pi(\hat{N})$ where $\pi$ is the projection $\TMO\to P_\D$.

Suppose we have a closed 2-form $\omega$ on $\TMO$ whose
characteristic distribution is spanned by $\Gamma$ and $\Delta$,
as in Theorem \ref{old}. Then $\omega=h_{ij}dx^i\wedge \phi^j$.
As we know, to determine whether $h$ is positive quasi-definite
at any $(x,y)\in\TMO$ we must consider
$\omega_{(x,y)}(\hlift{v},\vlift{v})$ as $v$ ranges over $T_xM$.
Moreover, $\omega$ determines a symplectic form $\Omega$ on
$P_\D$, according to Theorem \ref{AP}.

\begin{prop}
Let $N$ be a two-dimensional
totally-geodesic submanifold of $M$, $\bar{N}=\pi(\hat{N})$.
Then for any $p\in\bar{N}$ and any $\xi,\eta\in T_p\bar{N}$,
$\Omega_p(\xi,\eta)=\pm\omega_{(x,y)}(\hlift{v},\vlift{v})$ where
$(x,y)\in\hat{N}$ with $\pi(x,y)=p$, for some $v\in T_xN$.
\end{prop}

\begin{proof}
Let $\gamma$ be a geodesic of $\Gamma$ in $N$ whose path
projects to $p$; set $x=\gamma(0)\in N$, $y=\dot{\gamma}(0)\in
T_xN$. Now let $\nu(t)$ be a vector field along $\gamma$
everywhere tangent to $N$ and independent of $\dot{\gamma}$;
then $\{\nu(t),\dot{\gamma}(t)\}$ is a basis of
$T_{\gamma(t)}N$. Let $\xi,\eta\in T_pP_\D$:\ then there are
Jacobi fields $\xi(t),\eta(t)$ along $\gamma$ corresponding to
$\xi$ and $\eta$:\ furthermore, the vector fields
$X=\hlift{\xi}+\vlift{(\nabla\xi)}$,
$Y=\hlift{\eta}+\vlift{(\nabla\eta)}$ satisfy
$\lie{\Gamma}X=\lie{\Gamma}Y=0$; and
$\omega_{(\gamma(t),\dot{\gamma}(t))}(X(t),Y(t))$ is constant
and equal to $\Omega_p(\xi,\eta)$. Since $N$ is totally
geodesic, if $\xi,\eta\in T_p\bar{N}$ then
$\xi(t),\eta(t)\in T_{\gamma(t)}N$. We can express $\xi(t)$ and
$\eta(t)$ in terms of the basis $\{\nu(t),\dot{\gamma}(t)\}$:\ say
$\xi(t)=a(t)\nu(t)\pmod{\dot{\gamma}(t)}$,
$\eta(t)=b(t)\nu(t)\pmod{\dot{\gamma}(t)}$. Then
\[
\nabla\xi=\dot{a}\nu+a\nabla\nu,\quad
\nabla\eta=\dot{b}\nu+b\nabla\nu\pmod{\dot{\gamma}(t)},
\]
and so
\begin{align*}
\Omega_p(\xi,\eta)&=\omega(X,Y)\\
&=\omega(\hlift{\xi}+\vlift{(\nabla\xi)},\hlift{\eta}+\vlift{(\nabla\eta)})\\
&=\omega(\hlift{\xi},\vlift{(\nabla\eta)})
-\omega(\hlift{\eta},\vlift{(\nabla\xi)})\\
&=\omega(a\hlift{\nu},\dot{b}\vlift{\nu}+b\vlift{(\nabla\nu)})
-\omega(b\hlift{\nu},\dot{a}\vlift{\nu}+a\vlift{(\nabla\nu)})\\
&=(a\dot{b}-b\dot{a})\omega(\hlift{\nu},\vlift{\nu}).
\end{align*}
But $\omega(X,Y)$ is constant along $\gamma$, so in the end
\[
\Omega_p(\xi,\eta)
=(a\dot{b}-b\dot{a})(0)\omega_{(x,y)}(\hlift{\nu(0)},\vlift{\nu(0)}).
\]
Now if $a(0)\dot{b}(0)-b(0)\dot{a}(0)=0$ then $\xi(t)$ and
$\eta(t)$ are linearly dependent, and so $\Omega_p(\xi,\eta)=0$.
Otherwise, one can scale $\nu$ to eliminate the overall scalar
factor:\ that is, set
\[
v=\frac{1}{\sqrt{|a(0)\dot{b}(0)-b(0)\dot{a}(0)|}}\nu(0).\qedhere
\]
\end{proof}

Now $h$ is everywhere positive quasi-definite if and only if for
every $(x,y)\in \TMO$ and for every $v\in T_x M$,
$\omega_{(x,y)}(\hlift{v},\vlift{v})\geq 0$, and
$\omega_{(x,y)}(\hlift{v},\vlift{v})= 0$ if and only if $v$ is a
scalar multiple of $y$.
Suppose that $\Gamma$ has the property that
for every $x\in M$ and every two-dimensional subspace of $T_xM$ there is
a totally-geodesic submanifold $N$ through $x$ with the given subspace
as its tangent space.  Then for every point $p\in P_D$ and every
two-dimensional subspace of $T_p P_\D$ there is a two-dimensional
submanifold $\bar{N}$ through $p$ with the given subspace as its
tangent subspace. We conclude therefore:

\begin{prop}

Let $\Gamma$ be such that for every $x\in M$ and every two-dimensional
subspace of $T_xM$ there is a totally-geodesic submanifold
through $x$ with the given subspace as its tangent space. If $h$ is everywhere
positive quasi-definite, the pull-back of $\Omega$ to any
submanifold $\bar{N}$ as above is non-vanishing (i.e.\ it is a volume form).
Conversely, if $\Omega$ has this property then either $h$ or
$-h$ is everywhere positive quasi-definite.
\end{prop}

The question arises, are there any spray spaces with this property ---
other than those covered by Hilbert's 4th problem, namely those
for which the paths are straight lines? For a two-dimensional
totally-geodesic submanifold $N$, with $x\in N$ and $y,v\in
T_xN$, $\Phi_y(v)\in T_xN$ also:\ that is, $\Phi_y(v)$ is a
linear combination of $y$ and $v$ (if $v$ is a multiple of $y$
then $\Phi_y(v)=0$). If this holds for all $y$ and all $v\in
T_{\tau(y)}M$ then the space must be isotropic:\
$\Phi^i_j=\lambda\delta^i_j+\mu_jy^i (=R^i_j)$. We don't know,
however, whether this is sufficient as well as necessary. But it is
well known that every isotropic space is projectively metrizable, see e.g.\ \cite{BM,hol}.

\section{Example}
The following example, which is an extension of Shen's circle example
from \cite{Shen}, was introduced in \cite{mult}.

Consider the projective class of the spray
\[
\Gamma=u\vf{x}+v\vf{y}+w\vf{z}
+\sqrt{u^2+v^2+w^2}\left(-v\vf{u}+u\vf{v}\right)
\]
defined on $T^\circ\R^3$. As we showed in \cite{mult}, the
geodesics of $\Gamma$ are spirals with axis parallel to the
$z$-axis, together with straight lines parallel to the $z$-axis
and circles in the planes $z=\mbox{constant}$. Evidently both
$\sqrt{u^2+v^2}=\mu$ and $w$ are constant; and therefore (or
directly) $\sqrt{u^2+v^2+w^2}=\lambda$ is also constant. The
geodesics are the solutions of $\ddot{x}=-\lambda\dot{y}$,
$\ddot{y}=\lambda\dot{x}$, $\ddot{z}=0$, which are
\[
x(t)=\xi+r\cos(\lambda t+\vartheta),\quad
y(t)=\eta+r\sin(\lambda t+\vartheta),\quad
z(t)=wt+z_0,
\]
where $\xi,\eta,r,\vartheta$ are constants, with
$w^2=\lambda^2(1-r^2)$. The initial point on the geodesic (the
point where $t=0$) is $(x_0,y_0,z_0)$ where
$x_0=\xi+r\cos\vartheta$, $y_0=\eta+r\sin\vartheta$. The
projections of the geodesics on the $xy$-plane are circles of
center $(\xi,\eta)$ and radius $r=\mu/\lambda$:\ note that
$0\leq r\leq 1$, the circle degenerating to a point when $r=0$.
For $w/\lambda\neq0,\pm1$ the geodesics are spirals, with axis
the line parallel to the $z$-axis through $(\xi,\eta,0)$. The
case $r=0$ corresponds to $w/\lambda=\pm1$ and the geodesics are
straight lines parallel to the $z$-axis (in both directions).
The case $r=1$ ($w=0$) gives circles of unit radius in the
planes $z=z_0$.

Consider the genuine spirals, that is,
take $r\neq 0$ and $w\neq 0$.  Note first that the circle which is the
spiral's projection on the $xy$-plane is always traversed
anticlockwise, though $z(t)$ may increase or decrease with increasing
$t$, depending on the sign of $w$.  Next, we may fix the origin of $t$
so that $z_0=0$:\ then $\vartheta$ determines the point on the circle
in the $xy$-plane where $t=0$.  Let us (in general) set
$w/\lambda=\nu$:\ then $\nu$ is constant with $-1\leq\nu\leq 1$, and
also is homogeneous of degree 0 as a function on $T^\circ\R^3$.  We can
eliminate $t$, to express the spiral paths ($\nu\neq 0$) as
\[
x=\xi+\sqrt{1-\nu^2}\cos(z/\nu+\vartheta),\quad
y=\eta+\sqrt{1-\nu^2}\sin(z/\nu+\vartheta).
\]
Then $(\xi,\eta,\nu,\vartheta)$ smoothly parametrize
the set of genuine spirals.

(However, it is not possible to parametrize smoothly the
full set of paths.)

We have a map $(x,y,z,u,v,w)\mapsto
(\xi,\eta,\nu,\vartheta)$ where
\[
\xi=x-v/\lambda,\quad
\eta=y+u/\lambda,\quad
\nu=w/\lambda,\quad
\vartheta=\arccos(v/\mu)-\lambda z/w.
\]
Then
\[
d\xi=dx-d(v/\lambda),\quad
d\eta=dy+d(u/\lambda),\quad
d\nu=d(w/\lambda).
\]
It is simplest to compute $d\vartheta$ from the implicit definition,
which can be written
\[
\cos(z/\nu+\vartheta)=v/\mu,\quad \sin(z/\nu+\vartheta)=-u/\mu,
\]
from which it follows that
\[
d\vartheta=-d(z/\nu)-\mu^{-2}(vdu-udv)
=-d(z/\nu)-
(\lambda/\mu^2)\big(vd(u/\lambda)-ud(v/\lambda)\big).
\]
The 1-forms $d\xi$, $d\eta$, $d\nu$ and $d\vartheta$ are
evidently independent.

Consider the 2-form $\Omega=d\xi\wedge d\eta+\nu d\nu\wedge
d\vartheta$. It is a symplectic form. The spiral paths through
$(x,y,z)$ map to the points $(\xi,\eta,\nu,\vartheta)$ for which
\[
\xi=x-\sqrt{1-\nu^2}\cos(z/\nu+\vartheta),\quad
\eta=y-\sqrt{1-\nu^2}\sin(z/\nu+\vartheta),
\]
where we treat $x$, $y$ and $z$ as constants. On the 2-manifold so
defined we have
\begin{align*}
d\xi&=\frac{\nu}{\sqrt{1-\nu^2}}\cos(z/\nu+\vartheta)d\nu
+\sqrt{1-\nu^2}\sin(z/\nu+\vartheta)\big(-(z/\nu^2)d\nu+d\vartheta\big)\\
d\eta&=\frac{\nu}{\sqrt{1-\nu^2}}\sin(z/\nu+\vartheta)d\nu
-\sqrt{1-\nu^2}\cos(z/\nu+\vartheta)\big(-(z/\nu^2)d\nu+d\vartheta\big),
\end{align*}
whence
\[
d\xi\wedge d\eta=-\nu d\nu\wedge d\vartheta.
\]
That is, every such 2-manifold is Lagrangian for $\Omega$.

We next compute the pull-back $\omega$ of
$\Omega$ to $T^\circ\R^3$. We do so by using the
formulas above for $d\xi$, $d\eta$ etc., but no longer treat
$x$, $y$ and $z$ as constants. We have
\begin{align*}
d\xi\wedge d\eta&=(dx-d(v/\lambda))\wedge(dy+d(u/\lambda))\\
&=dx\wedge dy+dx\wedge d(u/\lambda)+dy\wedge d(v/\lambda)
+d(u/\lambda))\wedge d(v/\lambda)).
\end{align*}
On the other hand,
\[
\nu d\nu\wedge d\vartheta=dz\wedge d(w/\lambda)
-\mu^{-2}d(w/\lambda)\wedge
\big((vw)d(u/\lambda)-(uw)d(v/\lambda)\big).
\]
But since $(u/\lambda)^2+(v/\lambda)^2+(w/\lambda)^2=1$
\[
ud(u/\lambda)+vd(v/\lambda)+wd(w/\lambda)=0,
\]
whence
\begin{align*}
wd(w/\lambda)\wedge d(u/\lambda)&
=vd(u/\lambda)\wedge d(v/\lambda),\\
wd(w/\lambda)\wedge d(v/\lambda)&
=-ud(u/\lambda)\wedge d(v/\lambda),
\end{align*}
and therefore
\[
d(w/\lambda)\wedge\big((vw)d(u/\lambda)-(uw)d(v/\lambda)\big)
=(u^2+v^2)d(u/\lambda)\wedge d(v/\lambda).
\]
Finally, we have
\[
\omega=
dx\wedge dy+dx\wedge d(u/\lambda)+dy\wedge d(v/\lambda)
+dz\wedge d(w/\lambda).
\]

The 2-form $\omega$ satisfies the conditions of Theorem
\ref{old} and Corollary \ref{oldcor}. Thus $\Gamma$ should admit
a pseudo-Finsler function.

In fact $\Gamma$ comes from the pseudo-Finsler function
\[
F(x,y,z,u,v,w)=\sqrt{u^2+v^2+w^2}+\onehalf yu-\onehalf xv.
\]
This is globally well defined but only locally a Finsler
function. A straightforward calculation confirms that its
Hilbert 2-form is $\omega$ (up to sign).

The function $F$ is a Finsler function, that is, is positive,
only for $x^2+y^2<4$. It is globally pseudo-Finsler. To obtain a
Finsler function in a neighbourhood of an arbitrary point
$(x_0,y_0,z_0)$ we can make a simple modification to
\[
\tilde{F}(x,y,z,u,v,w)=\sqrt{u^2+v^2+w^2}
+\onehalf(y-y_0)u-\onehalf(x-x_0)v;
\]
this is positive for $(x-x_0)^2+(y-y_0)^2<4$. Note that it
differs from $F$ by a total derivative.

The planes $z=\mbox{constant}$ are totally-geodesic
submanifolds. Indeed, if we denote by $N$ any such plane then
$w=0$ on the corresponding submanifold $\hat{N}$ of
$T^\circ\R^3$, and the restriction of $\Gamma$ to $\hat{N}$ is
the spray
\[
u\vf{x}+v\vf{y}-v\sqrt{u^2+v^2}\vf{u}+u\sqrt{u^2+v^2}\vf{v}
\]
of Shen's circle example. We consider this as a spray defined on
$T^\circ\R^2$. It has for its geodesics all circles
in $\R^2$ of radius 1, traversed counter-clockwise. The path
space is smoothly parametrized by the coordinates $(\xi,\eta)$
of the circles' centres.

Again, this spray is locally projectively metrizable. One local
Finsler function is the restriction of the one given for the
spiral example, namely
\[
F(x,y,u,v)=\sqrt{u^2+v^2}+\onehalf yu-\onehalf xv.
\]
A straightforward calculation leads to its Hilbert 2-form:
\[
d\theta=-dx\wedge dy+\frac{1}{\mu^3}(vdu-udv)\wedge(vdx-udy).
\]
But this is just $-d\xi\wedge d\eta$. So in this case there is a
globally-defined path space equipped with a global symplectic
form. The Hilbert 2-form passes to the path space and coincides
with this symplectic form there. Moreover, it does so globally,
despite the fact that $F$ is only locally defined as a Finsler
function (though again it is global as a pseudo-Finsler
function).

\subsubsection*{Acknowledgements}
The first author is a Guest Professor at Ghent University:\ he
is grateful to the Department of Mathematics for its
hospitality. The second author is a Postdoctoral Fellow of the
Research Foundation -- Flanders (FWO). The third author
acknowledges the support of grant no.\ 201/09/0981 for Global
Analysis and its Applications from the Czech Science Foundation.

This work is part of the {\sc irses} project {\sc geomech} (nr.\
246981) within the 7th European Community Framework Programme.

We thank  the referee for drawing our attention to references \cite{Besse} and \cite{Ferrand}.

\subsubsection*{Address for correspondence}
M.\ Crampin, 65 Mount Pleasant, Aspley Guise, Beds MK17~8JX, UK\\
m.crampin@btinternet.com

\end{document}